\documentclass[11pt, a4paper]{amsart}
\usepackage{amssymb, array, amsmath, amscd, pdfpages, enumerate, amsthm, fixltx2e, setspace, hyperref}

\newcommand{\dd}{\mathrm{d}}

\newcommand{\CC}{\mathbb{C}}

\newcommand{\ZZ}{\mathbb{Z}}

\newtheorem{thm}{Theorem}[section]

\theoremstyle{definition}

\numberwithin{equation}{section}

\begin{document}

\title{Asymptotic behaviour in the robot rendezvous problem}
\author{Lassi Paunonen}
\address{Department of Mathematics, Tampere University of Technology, PO.\ Box 553, 33101 Tampere, Finland}
\email{lassi.paunonen@tut.fi}

\author{David Seifert}
\address{St John's College, St Giles, Oxford\;\;OX1 3JP, United Kingdom}
\email{david.seifert@sjc.ox.ac.uk}

\begin{abstract}
This paper presents a natural extension of the results obtained by Feintuch and Francis in \cite{FeiFra12,FeiFra12b} concerning  the so-called \emph{robot rendezvous problem}. In particular, we revisit a known necessary and sufficient condition for convergence of the solution in terms of Ces\`aro convergence of the translates $S^k x_0$, $k\ge0$, of the sequence $x_0$ of initial positions under the right-shift operator $S$, thus shedding new light on questions left open in \cite{FeiFra12,FeiFra12b}. We then present a new proof showing that a certain stronger ergodic condition on $x_0$  ensures that the corresponding solution converges to its limit at the optimal rate $O(t^{-1/2})$ as $t\to\infty$. After considering a natural two-sided variant of the robot rendezvous problem already studied in \cite{FeiFra12} and in particular proving a new quantified result in this case,  we conclude by relating the robot rendezvous problem to a more realistic model of vehicle platoons.
 \end{abstract}


\maketitle

\section{Introduction}\label{intro} 
Consider a situation in which there are countably many robots (or perhaps ants, beetles, vehicles etc.), indexed by the integers $\ZZ$, which at each time $t\ge0$ occupy the respective positions $x_k(t)$, $k\in\ZZ$, in the complex plane. Suppose moreover that, for each $k\in\ZZ$ and each time $t\ge0$, robot $k$ moves in the direction of robot $k-1$ with speed equal to their separation, so that
\begin{equation}\label{eq:rob_sys}
\dot{x}_k(t)=x_{k-1}(t)-x_k(t),\quad k\in\ZZ,\;t\ge0.
\end{equation}
We propose to investigate whether all of the
robots necessarily converge to a mutual meeting, or \emph{rendezvous}, point as  $t\to\infty$, that is to say whether there exists $c\in \CC$ such that $x_k(t)\to c$ as $t\to\infty $ uniformly in $k\in\ZZ$. 

The problem is a natural extension of the corresponding question for finitely many robots, and in the finite case it is a simple matter to show that all robots converge exponentially fast to the centroid of their initial positions. However, since the actual rate of exponential convergence tends to zero as the size of the system grows this leaves open the question whether in the infinite case one should expect any rate of convergence, or even convergence for all initial constellations. Indeed, it was shown in \cite{FeiFra12,FeiFra12b} that in the infinite setting there exist initial configurations of the robots which do not lead to convergence. The aim of this note is to revisit and extend a recent result due to the authors \cite{PauSei15} giving a complete and simple characterisation of which initial configurations do and which do not lead to convergence. 
Loosely speaking, we show 
that the robots converge to the centroid of their initial positions whenever this is well-defined in a suitable sense, and do not converge otherwise. 
In addition, we present a 
detailed description of the rates of convergence of the robots. Thus our paper serves to further elucidate the similarities and differences between large finite systems and infinite systems. For further discussion of the relation between finite and infinite systems of the general kind considered here, see for instance \cite{CuIfZwa09}.
 
Our approach is based on the asymptotic theory of $C_0$-semigroups and elements of ergodic theory, and the paper is organised as follows. Our first main result, giving a characterisation of those initial configurations leading to convergent solutions of the robot rendezvous problem, is presented in Section~\ref{sec:good}. In Section~\ref{sec:quant} we present  a new proof of a quantified result from \cite{PauSei15}, which provides an optimal estimate of the rate of convergence for initial configurations satisfying a certain condition, and in Section~\ref{sec:symm} we show how similar techniques  lead to a new quantified result in a natural two-sided variant of the robot rendezvous problem considered in \cite{FeiFra12}. We conclude in Section~\ref{sec:ext} by  describing a more realistic model which is representative of the general framework studied in depth in \cite{PauSei15}.

\section{Chararterising `good' initial constellations}\label{sec:good}

We begin by introducing some preliminary notions. Let $\ell^\infty(\ZZ)$ denote the space of doubly infinite sequences $(x_k)$ satisfying $\sup_{k\in\ZZ}|x_k|<\infty$, endowed with the supremum norm 
$$\|(x_k)\|=\sup_{k\in\ZZ}|x_k|,\quad (x_k)\in\ell^\infty(\ZZ).$$
 Since we are interested in convergence of the solution $x(t)=(x_k(t))$, $t\ge0$, with respect to the norm of $\ell^\infty(\ZZ)$, it is natural to assume that the initial constellation  $x_0=(x_k(0))$ is an element of $\ell^\infty(\ZZ)$, and we make this assumption throughout. We let $S$ denote the right-shift operator on $\ell^\infty(\ZZ)$, so that $S(x_k)=(x_{k-1})$ for all $(x_k)\in\ell^\infty(\ZZ)$.
 
 We say that an initial constellation $x_0$ in the robot rendezvous problem is \emph{good} if there exist $c_k\in\CC$, $k\in\ZZ$, such that the solution $x(t)$, $t\ge0$, of \eqref{eq:rob_sys} satisfies 
 $$\sup_{k\in\ZZ}|x_k(t)-c_k|\to0,\quad t\to\infty.$$
In the finite case all initial constellations are good, and the  robots all converge to the centroid of their initial positions. The following result shows that in the infinite robot rendezvous problem an initial constellation $x_0$ is good if and only if the translates $S^kx_0$, $k\ge1$, under the right-shift operator $S$ are Ces\`{a}ro summable with respect to the norm of $\ell^\infty(\ZZ)$, and that in this case the solution $x(t)$ of \eqref{eq:rob_sys} converges to this Ces\`aro limit, which is necessarily a constant sequence,  as $t\to\infty$. The result was originally obtained in \cite[Theorem~6.1]{PauSei15} as a consequence of a more general result with a lengthy proof. Here we give a short and direct proof combining the main result of Feintuch and Francis with elementary facts from ergodic theory.
 
 \begin{thm}\label{robot_thm}
 In the robot rendezvous problem \eqref{eq:rob_sys}, an initial constellation $x_0=(x_k(0))$ is good if and only if there exists $c\in \CC$ such that
 \begin{equation}\label{Cesaro}
 \sup_{k\in\ZZ}\bigg|\frac{1}{n}\sum_{j=1}^nx_{k-j}(0)-c\bigg|\to0,\quad n\to\infty,
 \end{equation}
 and if this is the case then 
  \begin{equation}\label{limit}
   \sup_{k\in\ZZ}|x_k(t)-c|\to0,\quad t\to\infty.
   \end{equation}
 \end{thm}
 
 \begin{proof}
 Let $T$ denote the $C_0$-semigroup generated by $S-I$, so that $T(t)=\exp(t(S-I))$ for $t\ge0$. Then the operators $T(t)$, $t\ge0$, are uniformly bounded in operator norm and the solution of  \eqref{eq:rob_sys} is given by $x(t)=T(t)x_0$, $t\ge0$. It follows from \cite[Theorem~3]{FeiFra12} that for initial constellations $x_0$ which lie in the range of $S-I$ we have $|x_k(t)|\to0$ as $t\to\infty$ uniformly in $k\in\ZZ$. Since the semigroup $T$ is uniformly bounded, the same conclusion holds for all initial constellations in the closure $Y$ of this range. Next observe that the kernel $Z$ of $S-I$ consists precisely of all constant sequences, and that such sequences are fixed by the semigroup. Let $X$ denote the space of all initial constellations in $\ell^\infty(\ZZ)$ which can be written (uniquely) as the sum of an element of $Y$ and an element of $Z$. Then by the above observations all elements of $X$ are good. By \cite[Proposition~4.3.1]{ABHN11} the elements of $X$ are also precisely those initial constellations $x_0$ for which the Ces\`aro means $$\frac1t\int_0^t T(s)x_0\,\dd s,\quad t>0,$$
 converge in the norm of $\ell^\infty(\ZZ)$  to a limit as $t\to\infty$. Since this is the case for any good initial constellation, $X$ in fact coincides with the set of all good constellations. Moreover, it is clear that if $x_0=y+z\in X$ with $y\in Y$ and $z\in Z$ being the constant sequence with entry $c\in\CC$, then \eqref{limit} holds. To finish the proof it suffices to observe that by \cite[Section~2.1, Theorem~1.3]{Kre85} the set $X$ also coincides with the set of all initial constellations $x_0$ for which \eqref{Cesaro} is satisfied.
 \end{proof}
 
It may be shown that condition \eqref{Cesaro} is satisfied for a wide range of initial constellations $x_0=(x_k(0))$, for instance whenever $x_k(0)=c+y_k$, $k\in\ZZ$, where $|y_k|\to0$ as $k\to\pm\infty$. In particular, the set of good initial constellation is stable under perturbations by  sequences which converge to zero. Thus  Theorem~\ref{robot_thm} strengthens  \cite[Lemma~2]{FeiFra12}. The result furthermore reveals the underlying reason for why  the construction given in \cite[Section~3.5]{FeiFra12} leads to an initial constellation $x_0$ which is not good and in particular gives a simple way of constructing other examples, for instance by taking $x_0=(x_k)$ to have entries $x_k=0$ for $k\ge0$ and, for $k<0$, alternating blocks of zeros and ones having lengths which increase at suitable rates. Perhaps the most important  contribution of Theorem~\ref{robot_thm} to the theory developed in \cite{FeiFra12} is the observation that the correct topology in which Ces\`aro convergence of translates needs to be studied is not the topology of convergence in each entry but the norm topology of $\ell^\infty(\ZZ)$. 
 
We observe in passing  that, even though it is argued in \cite{FeiFra12,FeiFra12b} that the above setting  for the robot rendezvous is the most realistic,  the problem can also be studied with initial constellations lying in $\ell^p(\ZZ)$, $1\le p<\infty$; see \cite[Theorem~6.1]{PauSei15}. The upshot is that for $1\le p<\infty$ the only possible rendezvous point is the origin, and that all initial constellations are good if $1<p<\infty$ but not when $p=1$. The latter statement is an immediate consequence of the well-known fact that the right-shift operator $S$ is mean ergodic on $\ell^p(\ZZ)$ if and only if $1<p<\infty$.
 
 \section{A quantified result}\label{sec:quant}
 
The following result is a quantified refinement of Theorem~\ref{robot_thm} and gives an estimate on the \emph{rate} of convergence for initial constellations $x_0$ which satisfy a slightly stronger condition  than \eqref{Cesaro}. The result was originally obtained in \cite[Theorem~6.1]{PauSei15}. However, whereas the proof given in \cite{PauSei15} relies on direct estimates involving Stirling's formula, we present here a new and more elegant proof. In what follows, given two sequences $(a_n)_{n\ge1}$ and $(b_n)_{n\ge1}$ of non-negative numbers, we write $a_n=O(b_n)$ as $n\to\infty$ if there exists a constant $C>0$ such that $a_n\le Cb_n$ for all sufficiently large $n\ge1$, and we use a similar notation for functions of a real variable.

 \begin{thm}\label{robot_quant}
 In the robot rendezvous problem \eqref{eq:rob_sys}, if $x_0=(x_k(0))$ is a good initial constellation such that 
 \begin{equation}\label{quant}
 \sup_{k\in\ZZ}\bigg|\frac{1}{n}\sum_{j=1}^nx_{k-j}(0)-c\bigg|=O\big(n^{-1}\big),\quad n\to\infty,
 \end{equation}
for some $c\in\CC$, then 
 $$\sup_{k\in\ZZ}|x_k(t)-c|=O\big(t^{-1/2}\big),\quad t\to\infty.$$
 \end{thm}
 
 \begin{proof}
 As in the proof of Theorem~\ref{robot_thm}, let $T$ denote the $C_0$-semigroup generated by $S-I$, and recall that the set of good constellations consists precisely of those initial constellations which can be written (uniquely) as the sum of a constant sequence and an element of the closure of the range of $S-I$. It follows from \cite[Theorem~5]{LinSine83} that condition \eqref{quant} in fact characterises those initial constellations $x_0$ which can be written as the sum of a constant sequence and an element of the \emph{range}, as opposed to the closure of the range, of $S-I$. Since constant sequences lie in the kernel of $S-I$ and consequently are fixed by the semigroup $T$,  the result will follow if we can establish that $\|T(t)(S-I)\|=O(t^{-1/2})$ as $t\to\infty$. Note first that, given $\varepsilon\in(0,1)$,  this property holds for $S-I$ and $T$ if and only if it holds for $\varepsilon(S-I)$ and the $C_0$-semigroup $T_\varepsilon$ generated by this operator. It is shown in  \cite[Theorem~1.2]{Dun08} that for the latter pair the required property is satisfied if and only if there exist $\beta\in(0,1)$ such that the operator 
$$Q_{\beta,\varepsilon}=\frac{\varepsilon(S-I)+1-\beta}{1-\beta}$$
is power-bounded. Since $Q_{1/2,1/2}=S$ is a contraction, and in particular power-bounded, the proof is complete. 
 \end{proof}
 
Examples of initial constellations $x_0=(x_k(0))$ satisfying condition~\eqref{quant} include sequences with $x_k(0)=c+y_k$, $k\in\ZZ$, where $\sum_{k\in\ZZ}|y_k|<\infty$. In particular, the set of initial constellations satisfying condition~\eqref{quant} is stable under perturbations by  sequences which are absolutely summable. Furthermore, it  follows from the results in \cite{PauSei15} not only that there cannot be a rate of convergence which holds for \emph{all} initial constellations $x_0$ but also that the rate $t^{-1/2}$ is optimal for those initial constellations $x_0$ which satisfy \eqref{quant}. This is in stark contrast to the case of finitely many robots, where all initial constellations lead to exponentially fast convergence to the centroid of the initial positions, albeit at decreasing exponential rates as the number of robots grows. As pointed out in the context of Theorem~\ref{robot_thm}, Theorem~\ref{robot_quant} also carries over to the $\ell^p$-case with $1\le p<\infty$; see \cite[Theorem~6.1]{PauSei15} for details.

 \section{The symmetric case}\label{sec:symm}
 
A natural variant of the robot rendezvous problem considered so far is the symmetric case in which each robot's motion is influenced by \emph{both} of its neighbours according to the ordinary differential equations
\begin{equation}\label{eq:rob_symm}
\dot{x}_k(t)=\frac{1}{2}\big(x_{k-1}(t)+x_{k+1}(t)\big)-x_k(t),\quad k\in\ZZ,\;t\ge0.
\end{equation}
As before, we follow \cite{FeiFra12} and consider this problem for initial constellations $x_0$ lying in $\ell^\infty(\ZZ)$. It was shown in \cite[Theorem~4]{FeiFra12} that the solution of \eqref{eq:rob_symm} satisfies
\begin{equation}\label{eq:nonquant}
\sup_{k\in\ZZ}|x_k(t)|\to0,\quad t\to\infty,
\end{equation}
whenever the vector $x_0$ lies in the range of $\frac12(S+S^{-1})-I$. Here $S^{-1}$, the inverse operator of $S$, is the left-shift operator on $\ell^\infty(\ZZ)$ given by $S^{-1}(x_k)=(x_{k+1})$.   
The following theorem presents an extended and quantified version 
of~\cite[Theorem~4]{FeiFra12}.
The result is an analogue of Theorems~\ref{robot_thm} and \ref{robot_quant}, giving also a characterisation of good initial constellations for the symmetric problem. 

 \begin{thm}\label{sym_thm}
 In the symmetric robot rendezvous problem \eqref{eq:rob_symm}, an initial constellation $x_0=(x_k(0))$ is good if and only if there exists $c\in \CC$ such that
 \begin{equation}\label{Cesaro_sym}
 \sup_{k\in\ZZ}\bigg|\frac{1}{n}\sum_{j=1}^n\frac{1}{2^j}\sum_{\ell=0}^j\binom{j}{\ell}x_{k-j+2\ell}(0)-c\bigg|\to0,\quad n\to\infty,
 \end{equation}
 and if this is the case then 
  \begin{equation}\label{conv_sym}
  \sup_{k\in\ZZ}|x_k(t)-c|\to0,\quad t\to\infty.
  \end{equation}
   Furthermore, if the convergence in \eqref{Cesaro_sym} is like $O(n^{-1})$ as $n\to\infty$, then the convergence in \eqref{conv_sym} is like $O(t^{-1})$ as $t\to\infty$.
 \end{thm}
 
 \begin{proof}
 The natural operator to consider is now $\frac12(S+S^{-1})-I$ rather than $S-I$.  Straightforward resolvent estimates show that this operator generates a bounded analytic $C_0$-semigroup, and it then follows from \cite[Theorem~3.7.19]{ABHN11} and the fact that the solution of \eqref{eq:rob_symm} is precisely the orbit of this semigroup that  
$$\sup_{k\in\ZZ}|x_k(t)|=O\big(t^{-1}\big),\quad t\to\infty,$$
whenever the initial constellation $x_0$ lies in the range of $\frac12(S+S^{-1})-I$. By an analogous argument to the one given in the proof of Theorem~\ref{robot_quant} together with a straightforward computation, decay in \eqref{Cesaro_sym} like $O(n^{-1})$ as $n\to\infty$  characterises those initial constellations $x_0$ which can be written (uniquely) as the sum of an element of this range and the constant sequence with entry $c$, which is fixed by the semigroup. The result now follows.
 \end{proof}

 \section{Further extensions}\label{sec:ext}

We mention in closing that Theorems~\ref{robot_thm} and \ref{robot_quant} are in fact special cases of a much more general theoretical apparatus developed in \cite{PauSei15}. As an example of the more realistic models that the general framework allows, suppose that each robot, or vehicle, $k\in\ZZ$ has associated with it not only a position $x_k$ but also a velocity $v_k$ and an acceleration $a_k$. We suppose that we can control the acceleration of each vehicle by means of a direct feedback control taking the form 
$$\dot{a}_k(t)=c_1y_k(t)+c_2v_k(t)+c_3a_k(t),\quad k\in\ZZ,\;t\ge0,$$
 where $y_k=x_k-x_{k-1}$ denotes the separation of vehicle $k$ from vehicle $k-1$ and where $c_1, c_2, c_3\in\CC$ are control parameters we are free to choose. It is natural to ask whether we can choose the control parameters in such a way that, as $t\to\infty$, all vehicles come to rest at a mutual meeting point. More generally, one might ask whether it is possible to steer the vehicles towards pre-specified target separations from one another, and questions of this kind have been studied in the control-theory literature for various types of vehicle platoons; see for instance \cite{CurIft09,JovBam05,ZwaFir13}. 

As is shown in \cite[Theorem~5.1]{PauSei15}, it is possible once again to characterise the good initial constellations in terms of a Ces\`aro condition (which, surprisingly, involves only the vehicles' initial deviations from the target separations, not their initial velocities or accelerations) and also to give a quantified result of the form of Theorem~\ref{robot_quant}. This time, however, the estimates are less straightforward and moreover \cite[Theorem~5.1]{PauSei15}  involves a logarithmic term in the estimate for the rate of convergence which was conjectured in \cite[Remark~5.2(a)]{PauSei15} to be unnecessary. It is shown in our recent paper \cite{PauSei16} how the argument outlined in the proof of Theorem~\ref{robot_quant} above can be extended to the more general setting of \cite{PauSei15}, thus in particular removing the logarithm  in the platoon model.

\bibliography{robots-reference}
\bibliographystyle{plain}
\end{document}